\documentclass{amsart}

\usepackage{amsmath,amsfonts,amsthm,amssymb}
\usepackage{graphicx,pinlabel, tikz}
\usetikzlibrary{positioning}
\usepackage{stmaryrd,comment}

\newcommand{\thicktilde}[1]{\mathbf{\tilde{\text{$#1$}}}}

\newtheorem{theorem}{Theorem}[section]
\newtheorem{corollary}[theorem]{Corollary}
\newtheorem{lemma}[theorem]{Lemma}
\newtheorem{proposition}[theorem]{Proposition}

\newtheorem*{mainthm}{Theorem~\ref{thm:main}}
\newtheorem*{binghandlebody}{Theorem~\ref{thm:binghandlebody}}

\newtheorem{question}[theorem]{Question} 
\newtheorem{claim}[theorem]{Claim} 
\theoremstyle{definition}
\newtheorem{definition}[theorem]{Definition}
\newtheorem{remark}[theorem]{Remark}

\newcommand{\pt}{\text{pt}}
\usepackage{xcolor}

\usepackage[margin=1.5in]{geometry}

\begin{document}

\begin{abstract}
We show that there exist split, orientable, 2-component surface-links in $S^4$ with non-isotopic splitting spheres in their complements. In particular, for non-negative integers $m,n$ with $m\ge 4$, the unlink $L_{m,n}$ consisting of one component of genus $m$ and one component of genus $n$ contains in its complement two smooth splitting spheres that are not topologically isotopic in $S^4\setminus L_{m,n}$. This contrasts with link theory in the classical dimension, as any two splitting spheres in the complement of a 2-component split link $L\subset S^3$ are isotopic in $S^3\setminus L$.
    \end{abstract}

\title[Non-isotopic splitting spheres in $S^4$]{Non-isotopic splitting spheres for a split link in $S^4$}

    \author[Mark Hughes]{Mark Hughes}
    \address{Brigham Young University\\Provo, UT, 84602 USA}
    \email{hughes@mathematics.byu.edu}
    
    \author[Seungwon Kim]{Seungwon Kim}
    \address{Sungkyunkwan University\\Suwon, Gyeonggi, 16419 Republic of Korea}
    \email{seungwon.kim@skku.edu}
    
    \author[Maggie Miller]{Maggie Miller}
    \address{{Stanford University\\Stanford, CA, 94305 USA}\hspace*{\fill}\linebreak \indent University of Texas at Austin\\Austin, TX, 78712 USA}
    \email{maggie.miller.math@gmail.com}
    

 \thanks{MH was supported by a grant from the NSF (LEAPS-MPS-2213295).  MM is supported by a Clay Research Fellowship and a Stanford Science Fellowship.}
\maketitle 

\section{Introduction}
In this paper, we study {\emph{splitting spheres}} in surface-link complements in $S^4$. Unless otherwise specified, statements are in the topological locally-flat category. 

\begin{definition}
Let $L=L_1\sqcup L_2$ be a  compact, codimension-2 submanifold of $S^{n}$ with two connected components $L_1$ and $L_2$. A {\emph{splitting sphere}} $S$ for $L$ is an embedded $S^{n-1}$ in $S^n\setminus L$ such that $L_1$ and $L_2$ lie in distinct components of $S^n\setminus S$.
\end{definition}

It is a consequence of the Schoenflies theorem in dimension three that any two splitting spheres $S_1,S_2$ for a 2-component link $L$ in $S^3$ are isotopic in the complement of $L$. Laudenbach showed that more generally, any two homotopic 2-spheres in a 3-manifold are isotopic \cite{laudenbach}. However, the situation is quite different in dimension four. For example, Iida--Konno--Mukherjee-Taniguchi \cite{konno} (following work of Donaldson and Wall) gave many examples of pairs of separating 3-spheres in 4-manifolds that are not smoothly isotopic. In this setting, the 3-spheres are topologically isotopic and the smooth obstruction necessitates $b_2^+$ of the ambient 4-manifold to be positive. However, avoiding any $b_2$ requirements, Budney--Gabai \cite{budneygabai} recently produced many homotopic yet smoothly non-isotopic essential 3-spheres in $S^1\times S^3$. 

As a consequence, Budney--Gabai showed that the unknotted 2-sphere in $S^4$ bounds 3-balls in $S^4$ that are not smoothly isotopic rel.\ boundary. Inspired by their work, in \cite{handlebodypaper} we proved (using very different methods) that for $g\ge 2$, the unknotted genus-$g$ surface in $S^4$ bounds smooth handlebodies that are homeomorphic rel.\ boundary but not topologically isotopic rel.\ boundary. In this paper, we adapt our previous techniques to prove the following theorem.

\begin{theorem}\label{thm:main}
    Let $m,n$ be non-negative integers with $m\ge 4$.
    Let $L = L_{m,n}$ be the 2-component unlink in $S^4$ with one component of genus $m$ and one component of genus $n$. There exist two splitting spheres $S_1,S_2$ for $L$ such that $S_1$ and $S_2$ are not isotopic in $S^4\setminus L$.
\end{theorem}

Here, an {\emph{unlink}} in $S^4$ is a split union of {\emph{unknots}}, while an {\emph{unknot}} in $S^4$ is a surface that bounds a 3-dimensional 1-handlebody in $S^4$. For each $g$, the genus-$g$ unknot is unique up to isotopy.  Note that the conclusion of Theorem~\ref{thm:main} greatly differs from those of \cite{budneygabai} or \cite{konno}: the 3-spheres we obtain are separating, not topologically isotopic, and are contained in a 4-manifold $(S^4\setminus\nu(L_{m,n})$) with $b_2^+=0$.

\begin{remark}\label{rem:smooth}
    In the proof of Theorem~\ref{thm:main}, the splitting 3-spheres $S_1,S_2$ can be taken to be smoothly embedded in the complement of the smooth unlink $L$, and each divide $S^4$ into two smooth 4-balls. Yet, the splitting spheres $S_1,S_2$ are not even topologically isotopic.
\end{remark}

\begin{remark}
    Theorem~\ref{thm:main} can also be extended to the case when the pair $(m,n)$ is in the set $\{(2,2),(2,3),(3,3)\}$ (see Remark \ref{rem:proof22}). 
    \end{remark}

Theorem~\ref{thm:main} may seem counterintuitive, since each of $S_1,S_2$ bound a 4-ball in $S^4$ and hence are isotopic in $S^4$ \cite{quinn}. However, to construct an isotopy from $S_1$ to $S_2$, we would generally first shrink the ball bounded by $S_1$ until $S_1$ and $S_2$ were disjoint, then isotope $S_1$ along the $S^3\times I$ cobounded by $S_1$ and $S_2$ (which is where we make use of \cite{quinn}). In the setting of a link complement, if $S_1,S_2\subset S^4\setminus L$ were disjoint then they would still cobound an $S^3\times I$ and would hence be isotopic. However, shrinking a 4-ball $B$ bounded by $S_1$ in order to separate $S_1$ and $S_2$ is now problematic, since one component of $L$ lies in $B$.

Moreover, we point out that while uniqueness of splitting spheres for 2-component links in $S^3$ follows from the 3-dimensional Schoenflies theorem, this argument does not extend to higher dimensions. In dimension three, we remove intersections between a pair of splitting spheres using a standard innermost circle argument: we surger one of the splitting spheres $S_1$ along a disk contained in the other sphere $S_2$ to split $S_1$ into two 2-spheres, one of which is a splitting sphere isotopic to $S_1$ that intersects $S_2$ in fewer circles. In contrast, two splitting 3-spheres in $S^4$ intersect in a surface, which we expect to include positive-genus components. Surgering one splitting sphere using such an intersection will split it into two 3-manifolds, but in general neither needs be a 3-sphere.

Using the proof of Theorem~\ref{thm:main}, we also prove the following link version of the main theorem of \cite{handlebodypaper}.

\begin{theorem}\label{thm:binghandlebody}
  Let $m,n$ be non-negative integers with $m\ge 4$. There exist two smooth, properly embedded links of handlebodies $H_1\sqcup H_2$ and $\thicktilde{H}_1\sqcup \thicktilde{H}_2$ in $B^5$, with $g(H_1)=g(\thicktilde{H_1})=m$ and $g(H_2)=g(\thicktilde{H}_2)=n$, such that the following are true:
    \begin{itemize}
        \item The 2-component links $H_1\sqcup H_2$ and $\thicktilde{H}_1\sqcup \thicktilde{H}_2$ are each smoothly boundary-parallel in $B^5$.
        \item The boundaries agree componentwise, i.e., $\partial H_1=\partial \thicktilde{H}_1, \partial H_2=\partial\thicktilde{H}_2$.
        \item For each $i=1,2$, the handlebodies $H_i,\thicktilde{H}_i$ are smoothly isotopic rel.\ boundary. 
        \item The 2-component links $H_1\sqcup H_2$ and $\thicktilde{H}_1\sqcup \thicktilde{H}_2$ are not topologically isotopic rel.\ boundary.
    \end{itemize}
\end{theorem}

    Given that the methods of this paper were established in our previous paper \cite{handlebodypaper} and were inspired by the results of Budney--Gabai \cite{budneygabai}, it would be reasonable to hope that the techniques of \cite{budneygabai} could be adapted to prove the existence of non-isotopic splitting spheres in the complement of a link of 2-spheres in $S^4$. 

\begin{question}\label{question:budneygabai}
Can one adapt the methods of Budney--Gabai to prove that the unlink of two spheres in $S^4$ has multiple splitting spheres? 
\end{question}

One approach to Question \ref{question:budneygabai} would be to take one 3-sphere $S_1$ to be standard, and obtain other 3-spheres from $S_1$ by ambient 2-handle surgery on either side of $S_1$, analogous to the barbells of \cite{budneygabai}. However, it seems to the present authors technically difficult to extend the isotopy obstruction used by Budney--Gabai from the original setting of the homotopy class of $\pt\times B^3$ in $S^1\times B^3$ to a homotopy class of separating 3-spheres in $S^4\setminus \nu(L_{0,0})\cong S^1\times B^3\#S^1\times B^3$.

\subsection*{Organization}
In Section \ref{sec:bing} we discuss Bing doubling of codimension-2 submanifolds, particularly in ambient dimensions four and five. In Section \ref{sec:simple}, we discuss simple 3-knots, which are 3-spheres embedded in $S^5$ whose complements have cyclic fundamental group. Finally in Section \ref{sec:construct}, we make use of the constructions in the previous two sections to produce two distinct splitting spheres in the complement of a trivial surface link, proving Theorems~\ref{thm:main} and \ref{thm:binghandlebody}.
\subsection*{Acknowledgement}
The work in this paper took place while the first and third authors were visiting the second author at Sungkyunkwan University (Suwon campus), Republic of Korea during February 2023.  

\section{Bing doubles}\label{sec:bing}
In order to produce interesting codimension-2 links, we make use of a well-known satellite operation, {\emph{Bing doubling}}. Since this well-known construction in classical knot theory is less well-known in higher dimensions, we review its definition here.

\begin{definition}\label{def:bingdouble}
    Let $K^{n-2}$ be a closed, orientable, codimension-2 submanifold of $S^n$ for some $n\ge 3$. Let $\Sigma\subset K$ be a separating $(n-3)$-sphere in $K$.  The {\emph{Bing double of $K$ along $\Sigma$}}, which we refer to as $B_\Sigma(K)$, is a 2-component codimension-2 link in $S^n$  obtained as follows. 
    \begin{enumerate}
        \item Start with two parallel disjoint copies $K',K''$ of $K$. In order to obtain these parallel copies, we choose a framing of $K$. Since the normal bundle of $K$ in $S^n$ is trivial, framings are in 1-to-1 correspondence with elements of $H^1(K;\mathbb{Z})$. We specifically choose the zero framing, which has the property that the parallel copies $K',K''$ are nullhomologous in $S^n\setminus K$. (In other words, we first find an $(n+1)$-manifold $M$ in $S^n$ bounded by $K$, and we take $K',K''$ to be parallel copies of $K$ that bound disjoint, parallel copies of $M$.)
        \item Let $V$ be a copy of $B^3\times S^{n-3}$ in $S^n$ intersecting $K'\sqcup K''$ in tubular neighborhoods $P'$ and $P''$ of pushed-off copies of $\Sigma$ in each of $K'$ and $K''$ respectively.
        \item Parameterize $V=B^3\times S^{n-3}$ so that \[\left(V,V\cap (P'\sqcup P'')\right)=\left(B^3\times S^{n-3},T_\infty\times S^{n-3}\right),\] where $(B^3,T_\infty)$ is the $\infty$-tangle (see Figure \ref{fig:bingdouble_def}).
        \item Delete $V=B^{3}\times S^{n-3}$ and replace it with $(B^3,T_1)\times S^{n-3}$, where $T_1$ is the $+1$-tangle (again, see Figure \ref{fig:bingdouble_def}).
    \end{enumerate}

Let $A,B$ denote the closures of the two components of $K\setminus\Sigma$. By construction, each component of $B_{\Sigma}(K)$ bounds a copy of $A\times I$ or $B\times I$. In particular, if $K$ is a sphere then each component of $B_{\Sigma}(K)$ bounds a ball (and hence is unknotted), though these balls are not disjoint and in general we do not expect the components of $B_{\Sigma}(K)$ to form an unlink (see Lemma~\ref{lemma:bingnotsplit}).
\end{definition}
\begin{figure}
\labellist
\pinlabel{$K$} at 0 200
\pinlabel{\textcolor{red}{$\Sigma$}} at 90 198
\pinlabel{$K'$} at 210 200
\pinlabel{$K''$} at 238 177
\pinlabel{\textcolor{red}{$V$}} at 301 202
\pinlabel{$B(K)$} at 0 90
\pinlabel{\textcolor{red}{$V$}} at 90 85
\pinlabel{$B(K)$} at 220 90
\endlabellist
    \includegraphics[width=110mm]{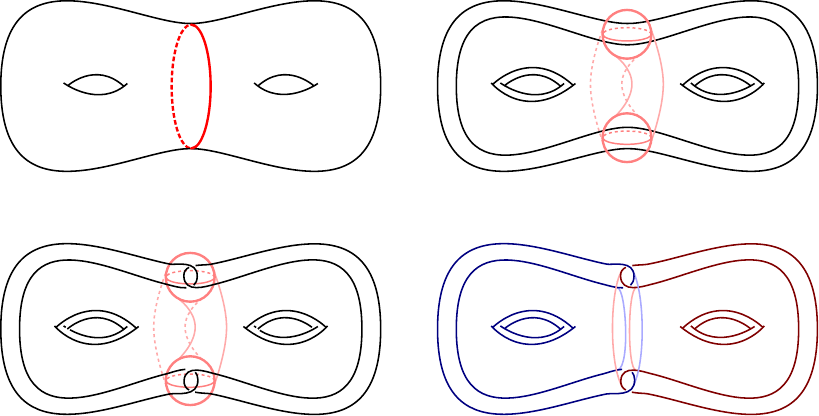}
    \caption{Top left: a schematic of a codimension-2 manifold $K^{n-2}$ of $S^n$ and a separating $(n-3)$-sphere $\Sigma$ in $K$. Top right: two parallel copies $K',K''$ of $K$ and a tubular neighborhood $V=B^3\times S^{n-3}$ of $\Sigma$ intersecting $K'\cup K''$ in $T_\infty\times S^{n-3}$ where $T_\infty$ is the $\infty$-tangle. Bottom left: we replace the intersection of $(K'\cup K'')\cap V$ with $T_1\times S^{n-3}$ where $T_1$ is the $+1$-tangle. Bottom right: the Bing double $B_\Sigma(K)$.}\label{fig:bingdouble_def}
\end{figure}
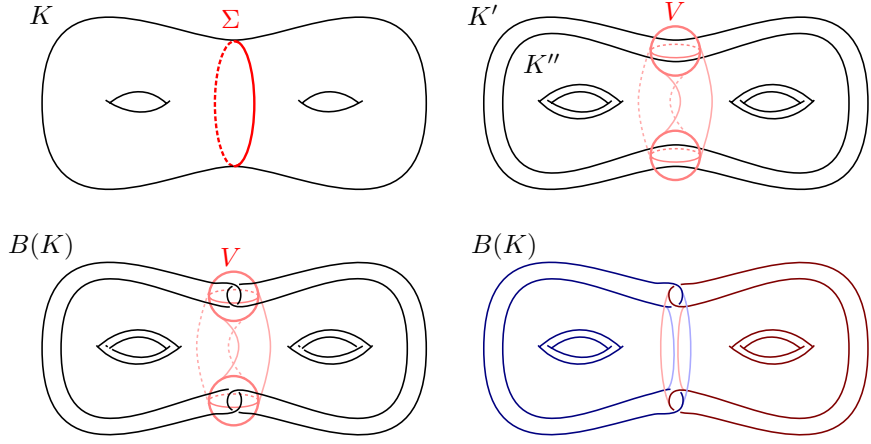

\begin{remark}
Let $P$ be a 3-sphere in $S^5$ intersecting an equatorial $S^4$ in a connected surface $S$, and let $\Sigma$ be a 2-sphere in $P$ intersecting $S$ in a circle $\gamma$. Then $B_\Sigma(P)$ intersects $S^4$ in a Bing double $B_\gamma(S)$.
\end{remark}

\begin{proposition}\label{prop:bingunknotisunlink}
    Let $S$ be a smoothly unknotted genus-$g$ surface in $S^4$. Then for any separating curve $\gamma$ in $S$, $B_\gamma(S)$ is a smooth 2-component unlink.
\end{proposition}

Note that it is obvious from Definition \ref{def:bingdouble} that each component of $B_{\gamma}(S)$ is unknotted for any surface $S$ in $S^4$, as $\gamma$ cuts $S$ into two 2-dimensional 1-handlebodies and hence each component of $B_{\gamma}(S)$ bounds a 3-dimensional 1-handlebody.

\begin{proof}[Proof of Proposition~\ref{prop:bingunknotisunlink}]
    Let $F_1, F_2$ be the closures of the two components of $S\setminus\gamma$. Let $\gamma'$ be any other separating curve on $S$ so that the closures $F'_1,F'_2$ of the components of $S\setminus\gamma'$ satisfy $g(F'_i)=g(F_i)$ for each $i$. 
    
    \begin{claim}There is a self-diffeomorphism $\phi$ of $S^4$ fixing $S$ setwise but with $\phi(\gamma)=\gamma'$.
    \end{claim}
    \begin{proof}[Proof of Claim]
    Let $(\alpha_1,\beta_1)$ be a symplectic basis of $H_1(F_1;\mathbb{Z})$ and $(\alpha_2,\beta_2)$ be a symplectic bases of $H_1(F_2;\mathbb{Z})$. Take all $\alpha_i,\beta_i$ curves to have trivial Rokhlin invariant (as in \cite{hirose}). Similarly, choose vanishing-Rokhlin symplectic bases $(\alpha'_1,\alpha'_2),(\beta'_1,\beta'_2)$ for $H_1(F'_1;\mathbb{Z})$ and $H_1(F'_2;\mathbb{Z})$ respectively. Let $\psi:S\to S$ be a surface automorphism with $\psi(\alpha_i)=\alpha'_i$ and $\psi(\beta_i)=\beta'_i$ for each $i$. Then up to isotopy, $\psi(\gamma)=\gamma'$. By Hirose \cite{hirose}, the map $\psi$ extends as a diffeomorphism over $S^4$.
    \end{proof}

In particular, there is a diffeomorphism $\phi:S^4\to S^4$ taking $(S,\gamma)$ to the pair $(S,\gamma')$ illustrated in Figure \ref{fig:bingunknot}. In Figure \ref{fig:bingunknot} we construct $B_{\gamma'}(S)$ and explicitly observe that this 2-component link is an unlink. Since $(S^4,B_{\gamma'}(S))\cong(S^4,B_\gamma(S))$, we conclude that $B_{\gamma}(S)$ is also an unlink.
\end{proof}

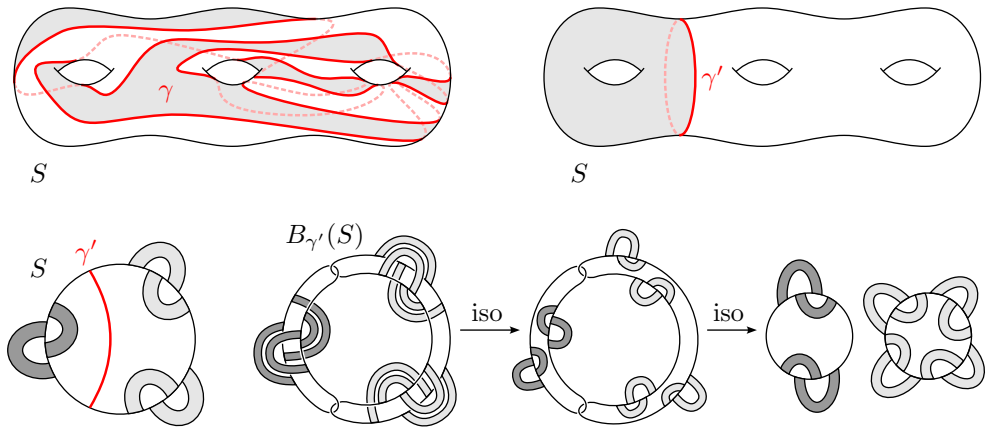
\begin{figure}
\labellist
\pinlabel{$S$} at 10 80
\pinlabel{\textcolor{red}{$\gamma$}} at 50 105
\pinlabel{$S$} at 180 80
\pinlabel{\textcolor{red}{$\gamma'$}} at 222 110
\pinlabel{$S$} at 10 50
\pinlabel{\textcolor{red}{$\gamma'$}} at 25 55
\pinlabel{$B_{\gamma'}(S)$} at 100 60
\pinlabel{iso} at 151 36
\pinlabel{iso} at 227 36
\endlabellist
    \includegraphics[width=130mm]{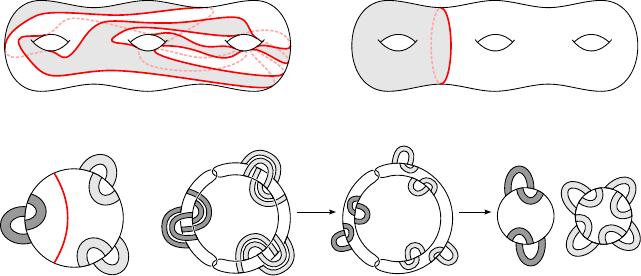}
    \caption{Top row: For $S$ an unknotted surface in $S^4$, there is a self-diffeomorphism $\phi$ of  $S^4$ fixing $S$ setwise and taking a separating curve $\gamma$ to a standard separating curve $\gamma'$. Taking $S$ to be contained in an equatorial 3-sphere, $\gamma'$ is the intersection of $S$ with an equatorial $S^2$ of the 3-sphere. We illustrate here the case that $\gamma$ cuts $S$ into a genus-1 and a genus-2 piece. Bottom row: we draw a banded unlink diagram of the Bing double $B_{\gamma'}(S)$ and observe that $B_{\gamma'}(S)$ is an unlink.}\label{fig:bingunknot}
\end{figure}

When $K^{n-2}$ is a codimension-2 sphere, then we can omit the separating sphere $\Sigma$ from the notation $B_\Sigma(K)$, since any two codimension-1 subspheres of $K$ are isotopic in $K$ (remember that we are working in the topological locally-flat category). In this setting, it is also a simple exercise to check that $B(K)$ is well-defined up to isotopy in a neighborhood of $K$.

\begin{remark}\label{rem:fourcopies}
    Let $K^{n-2}$ be a codimension-2 sphere in $S^n$ for $n\ge 3$. Then the Bing double of $K$ can also be obtained by the following procedure.

Let $K_1,K_2,K_3,K_4$ be disjoint parallel copies of $K$. (If $n=3$ then ensure these copies are parallel with the 0-framing; if $n>3$ then $H^1(K;\mathbb{Z})=0$ so there is no ambiguity.) Give $K_1,K_2$ parallel orientations and $K_3,K_4$ the opposite orientation. Surger $K_1\sqcup K_2\sqcup K_3\sqcup K_4$ along two arcs as specified in Figure \ref{fig:bingfourcopies}, taking framing on the arcs that are compatible with the relative orientations of $K_1,K_3$ and $K_2,K_4$.
\end{remark}

\begin{figure}
\labellist
\pinlabel{$K$} at 6 29
\pinlabel{\textcolor{red}{framed arcs}} at 79 63
\pinlabel{$K$} at 58.5 29
\pinlabel{$K$} at 125.5 29
\pinlabel{$K$} at 191 29
\endlabellist
\vspace{.1in}
    \includegraphics[width=130mm]{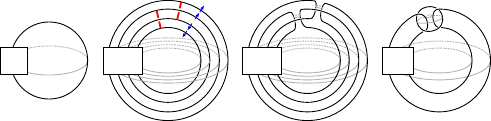}
    \caption{First: A schematic of a codimension-2 sphere $K$ in $S^n$. Second: Four parallel copies of $K$ with indicated orientations, and two framed arcs connecting different pairs of copies of $K$. Third: We surger the four parallel copies of $K$ along the two arcs to obtain a link of two spheres, each of which is unknotted. Fourth: The link is isotopic to the Bing double $B(K)$.}\label{fig:bingfourcopies}
\end{figure}
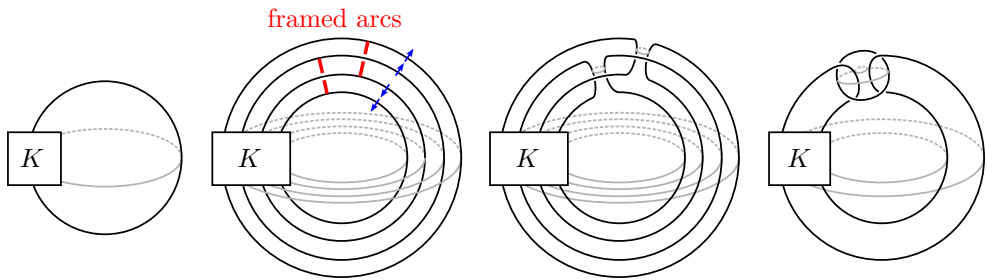

Framings of an arc in dimension $n\ge 4$ are in bijection with $\pi_1(SO(n-1))=\mathbb{Z}/2\mathbb{Z}$; requiring that the surgeries respect the orientations on $K_1,K_2,K_3,K_4$ determines the framings on the two arcs, and thus the Bing double $B(K)$ is well-defined up to isotopy when $K^{n-2}$ is a codimension-2 sphere with $n\ge 4$.

Additionally, it is useful to note that $B(K)$ can be obtained by a certain satellite operation.

\begin{remark}\label{rem:satellite}
Let $K^{n-2}$ be a codimension-2 sphere in $S^n$ for $n\ge 3$. Then the Bing double of $K$ can be obtained by the following satellite operation (see Figure \ref{fig:bing_satellite}).

Let $U_1\sqcup U_2$ be a 2-component unlink of $(n-2)$-spheres in $S^n$. Then $\pi_1(S^n\setminus (U_1\sqcup U_2))$ is a free group generated by meridians $\mu_1$, and $\mu_2$ for $U_1,U_2$ respectively. Let $C$ be a curve in the complement of $U_1\sqcup U_2$ freely representing $\mu_1\mu_2\mu_1^{-1}\mu_2^{-1}$.

Now set $(X,B)=(S^{n}\setminus \nu(C), U_1\sqcup U_2)$. The Bing double $B(K)$ of $K$ is then obtained by \[(S^n,B(K))=\left((S^n\setminus\nu(K)) \cup X, B\right).\]

If $n=3$, then the union above must be chosen so that a meridian of $C$ is glued to a 0-framed longitude of $K$. If $n\ge 4$ the resulting link does not depend on the gluing map  (\cite[Corollary 9.3]{stallings} for $n\ge 5$, \cite{laudenbachpoenaru} for $n=4$).
\end{remark}

\begin{figure}
    \centering
    \labellist
    \pinlabel{$U_1$} at 10 37
    \pinlabel{$U_2$} at 40 37
    \pinlabel{\textcolor{red}{$C$}} at 25 35
    \pinlabel{\textcolor{red}{$S^n\setminus\nu(C)$}} at 85 50
    \pinlabel{$B$} at 69 12
    \endlabellist
    \vspace{.1in}
    \includegraphics[width=100mm]{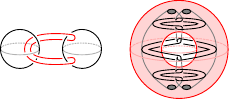}
    \caption{When $K^{n-2}$ is a codimension-2 sphere in $S^n$, the Bing double of $K$ is obtained by the illustrated satellite operation. Left: $U_1,U_2$ form a codimension-2 unlink in $S^n$ and $C$ is a curve representing $\mu_1\mu_2\mu_1^{-1}\mu_2^{-1}$ in the complement of $U_1\sqcup U_2$, where $\mu_i$ is a meridian of $U_i$. Right: the pattern $(X,B)=(S^n\setminus\nu(C),U_1\sqcup U_2)$.}
    \label{fig:bing_satellite}
\end{figure}
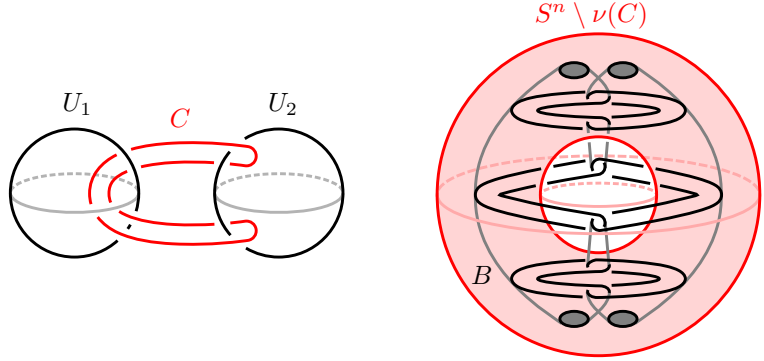

Both Remarks \ref{rem:fourcopies} and \ref{rem:satellite} contain standard phrasings of the definition of the Bing double of a classical knot in $S^3$, extended to hold in general dimension.

\section{Simple 3-knots}\label{sec:simple}

In this section, we study (smooth/topological) {\emph{3-knots}}, which are (smooth/locally-flat) submanifolds of $S^5$ that are homeomorphic to $S^3$. We say that a 3-knot $P$ is {\emph{simple}} if $\pi_1(S^5\setminus\nu(P))\cong\mathbb{Z}$. A 3-knot is {\emph{unknotted}} if it bounds a 4-ball in $S^5$ (with no mention of category necessary, as any smooth 3-knot that bounds a locally flat 4-ball also bounds a smooth 4-ball \cite[Theorem 1.1]{shaneson}).

While the unknotted 3-knot is simple, there are also many nontrivial simple 3-knots. Consider the following lemma, which appears in \cite{handlebodypaper} and is mostly a consequence of Hirose's work on extending surface automorphisms from an unknotted surface to all of $S^4$ \cite{hirose}.

\begin{lemma}\label{simplelemma}
    Every smooth 2-knot $K$ arises as the cross-section of a smooth simple 3-knot. If $K$ can be turned into a smoothly unknotted genus-$g$ surface via surgery along $g$ arcs, then the height function $h$ on $S^5$ can be taken to restrict to this simple 3-knot as a Morse function with one index-0 point, $2g$ index-1 points, $2g$ index-2 points, and one index-3 point.
\end{lemma}

Lemma~\ref{simplelemma} is particularly interesting in the case that $K$ is smooth and not doubly slice (e.g., when $K$ is the $2$-twist spun trefoil \cite{stoltzfus}), as then we automatically know that $K$ is the cross-section of some nontrivial simple 3-knot.

The existence of nontrivial simple 3-knots makes the following lemma surprisingly nontrivial to prove.

\begin{lemma}\label{lemma:bingnotsplit}
    Let $P$ be a nontrivial 3-knot. Then the Bing double $B(P)$ of $P$ is not split.
\end{lemma}
\begin{proof}

The components of $B(P)$ are each unknotted, so if $B(P)$ is split then $B(P)$ is an unlink. Thus if $B(P)$ is split, $S^5\setminus\nu(B(P))\simeq S^1\vee S^1\vee S^4$. In particular, if $B(P)$ is split then $\pi_1(S^5\setminus \nu(B(P)))$ is free and $\pi_2(S^5\setminus \nu(B(P)))$ is trivial. 

As observed in Remark \ref{rem:satellite},
the complement $S^5\setminus \nu(B(P))$ decomposes as a union of a manifold $X_1\cong S^5\setminus\nu(P)$ and 
 a manifold $X_2$ homeomorphic to the complement in $S^5$ of a generalized Borromean rings of two 3-spheres and a circle. The manifolds $X_1$ and $X_2$ are glued along an $S^1\times S^3$ boundary component (which in $X_2$ is the boundary of the deleted tubular neighborhood of a circle). Let $U_1, U_2$ be small tubular open neighborhoods of $X_1, X_2$ respectively. The Seifert--van Kampen theorem applied to $\{U_1,U_2\}$ yields the following pushout.
$$\begin{tikzpicture}[node distance=1cm, auto]
\node (O) {$\mathbb{Z}$};
  \node (A) [right=.2cm of O] {$\pi_1(S^1\times S^3,\ast)$};
  \node (B) [right=1cm of A] {$\pi_1(S^5\setminus\nu(P))$};
  \node (M) [right=.2cm of B] {$\langle\langle\mu\rangle\rangle$};
  \node (C) [below=1.5cm of A] {$\pi_1(X_2)$};
  \node (X) [left=.2cm of C] {$\langle a,b\rangle$};
  \node (D) [below=1.5cm of B] {$\pi_1(S^5\setminus\nu(B(P)))$};
   \draw[->] (A) to node [above] {$1\mapsto\mu$} (B);
   \draw[->] (A) to node [left] {\rotatebox{-90}{\rotatebox{90}{$1$} $\mapsto$ \rotatebox{90}{\hspace{-.2in}$ aba^{-1}b^{-1}$}}} (C);
   \draw[->] (B) to node [right] {} (D);
   \draw[->] (C) to node [above] {} (D);
   \path ([shift={(0,-.2)}]O.east) to node {{\scalebox{1}[1]{$=$}}} ([shift={(0,-.2)}]A.west);
   \path ([shift={(0,-.2)}]B.east) to node {{\scalebox{1}[1]{$=$}}} ([shift={(0,-.2)}]M.west);
   \path ([shift={(0,-.2)}]X.east) to node {{\scalebox{1}[1]{$=$}}} ([shift={(0,-.2)}]C.west);
\end{tikzpicture}$$
If $\pi_1(S^5\setminus\nu(P))\not\cong\mathbb{Z}$, then $\pi_1(S^5\setminus\nu(B(P))$ is not free and we conclude that $B(P)$ is not split.

We now restrict to the case that $\pi_1(S^5\setminus\nu(P))\cong\mathbb{Z}$, i.e., $P$ is simple. We will show that $\pi_2(S^5\setminus B(P))\not\cong 0$.

By Kreck--Su \cite{kreck}, if a 3-knot $Q\subset S^5$ is such that $\pi_1(S^5\setminus Q)\cong\mathbb{Z}$ and $\pi_2(S^5\setminus Q)=0$, then $S^5\setminus\nu(Q)$ is homeomorphic to $S^1\times B^4$, so by e.g. \cite[Corollary 9.3]{stallings} the 3-knot $Q$ bounds a ball in $S^5$. 
Since $\pi_1(S^5\setminus \nu(P))\cong\mathbb{Z}$ and $P$ does {\emph{not}} bound a ball, we conclude that $\pi_2(S^5\setminus \nu(P))\neq 0$.

We will make use of 
Althoen's analogue of the Seifert van Kampen theorem for $\pi_2$ {\cite{althoen}}. This theorem applies to an open cover $\{V_1, V_2\}$ of a path-connected topological space when $V_1, V_2$, and $V_1\cap V_2$ are all path-connected and the maps induced by inclusion from $\pi_1(V_1\cap V_2)\to \pi_1(V_i)$ are injective for both $i=1,2$. We have already seen that this is the case for the open cover $\{U_1, U_2\}$ of $S^5\setminus\nu(B(P))$.

We thus obtain the following pushout diagram, with $\ast$ a basepoint in $U_3:=U_1\cap U_2\simeq S^1\times S^3$ and indexing sets $I_j=\pi_1(S^5\setminus\nu(B(P)),\ast)/\pi_1(U_j,\ast)$.

$$\begin{tikzpicture}[node distance=2cm, auto]
\node (O) {$0$};
\node (O2) [right=.5cm of O] {$\bigoplus_{\mathbb{N}}\pi_2(S^1\times S^3)$};
  \node (A) [right=.5cm of O2] {$\bigoplus_{I_3}\pi_2(U_3,\ast)$};
  \node (B) [right=1cm of A] {$\bigoplus_{I_1}\pi_2(U_1,\ast)$};
  \node (X) [right=.5cm of B] {$\bigoplus_{\mathbb{N}}\pi_2(S^5\setminus\nu(P))$};
  \node (C) [below=1cm of A] {$\bigoplus_{I_2}\pi_2(U_2,\ast)$};
  \node (D) [below=1cm of B] {$\pi_2(S^5\setminus\nu(B(P)),\ast)$};
   \draw[->] (A) to node [above] {$g$} (B);
   \draw[->] (A) to node [above] {} (C);
   \draw[->] (B) to node [right] {$f$} (D);
   \draw[->] (C) to node [above] {} (D);
   \path ([shift={(0,-.2)}]O.east) to node {{\scalebox{2}[1]{$=$}}} ([shift={(0,-.2)}]O2.west);
   \path ([shift={(0,-.2)}]O2.east) to node {{\scalebox{2}[1]{$=$}}} ([shift={(0,-.2)}]A.west);
   \path ([shift={(0,-.2)}]B.east) to node {{\scalebox{2}[1]{$=$}}} ([shift={(0,-.2)}]X.west);
\end{tikzpicture}$$

Because $g\equiv0$, the map $f$ is an injection. Since $\pi_2(S^5\setminus\nu(P))\not\cong 0$ we conclude that $\pi_2(S^5\setminus\nu(B(P)))\not\cong0$ and hence $B(P)$ is not split.
\end{proof}

\begin{proposition}\label{prop:bingmorse}
Let $K$ be a smooth 2-knot with the property that $K$ can be turned into a smoothly unknotted torus via surgery along some arc.   There is a smooth link $P_1\cup P_2$ of unknotted 3-spheres in $S^5$ so that $(P_1\cup P_2)\cap S^4=B(K)=B_1\sqcup B_2$ (with $P_i\cap S^4=B_i)$, and the height function $h$ on $S^5$ restricts to each of $P_1,P_2$ as follows.
\begin{itemize}
\item The function $h|_{P_1}$ is Morse with one local minimum, four index-1 critical points, four index-2 critical points, and one local maximum.
    \item The function $h|_{P_2}$ is Morse with one local minimum, one local maximum, and no other critical points.
\end{itemize}
Moreover, $P_1\cup P_2$ is isotopic to the Bing double $B(P)$ of $P$ for $P$ some simple 3-knot with cross-section $K$.
\end{proposition}

\begin{proof}
Choose a curve $\gamma$ on $K$ and isotope $B(K)$ to agree with $B_\gamma(K)$. The curve $\gamma$ separates $K$ into disks $D_1$ and $D_2$ so that $B_1,B_2$ are respectively doubles of $D_1,D_2$.  Let $H$ be a 3-dimensional 1-handle with ends on $K$ so that surgering $K$ along $H$ yields an unknotted torus $T$. Isotope $H$ so that the ends of $H$ lie in $D_1$.  By Hirose \cite{hirose}, $T$ bounds a solid torus $V$ so that there is a curve on $T$ bounding a disk in $V$ that is geometrically dual to the core of $H$. As in Lemma~\ref{simplelemma}, we produce a simple 3-knot $P$ with cross-section $K$ by specifying the intersections of $P$ with various subsets of $S^4\times[-1,1]$ as follows (see left of Figure \ref{fig:bingdouble_sliceball}).

\begin{itemize}
 \item $(S^4\times[-1,0], S^4\times[-1,0]\cap P)$ is the mirror of $(S^4\times[0,1], S^4\times[0,1]\cap P)$.
    \item $P\cap (S^4\times [0,1/3))=K\times[0,1/3)$,
    \item $P\cap (S^4\times\{1/3\})=(K\cup H)\times\{1/3\}$,
    \item $P\cap(S^4\times (1/3,1))=T\times(1/3,1)$,
    \item $P\cap(S^4\times\{1\}))=V\times\{1\}$.
\end{itemize}

 Note that we could smooth $P$ so that $h|_P$ is Morse with one local minimum, two index-1 critical points, two index-2 critical points, and one local maximum.\footnote{The resulting 3-manifold $P$ thus admits a handle decomposition with one 0-handle, two cancelling 1/2-handle pairs, and one 3-handle, hence is a 3-sphere. Since $h|_P$ has a unique local minimum, $\pi_1(S^5\setminus\nu(P))$ is cyclic and hence $P$ is simple.}

Now let $\Sigma$ be a separating 2-sphere in $P$ as in Figure \ref{fig:bingdouble_sliceball}. That is, $\Sigma\cap h^{-1}(0)=C$, both index-1 points of $h|_P$ lie in the same component of $P\setminus \Sigma$ as does $D_1$, and $h|_\Sigma$ has one local minimum and one local maximum, which coincide with the local minimum and maximum of $h|_P$. We can then take $\Sigma\cap (V\times\{1\})$ to be a disk. Let $V_1,V_2$ be the genus-2 handlebody and 3-ball (respectively) each obtained by taking one component of $(V\times\{1\})\setminus\Sigma$ and doubling along the intersection with $\Sigma$. Let $H_1,H_2$ denote parallel copies of the 3-dimensional 1-handle $H$, with ends on distinct sheets of $B_1$. Let $B_1'$ denote the result of surgering $B_1$ along $H_1\sqcup H_2$ (note that $B_1'$ bounds a copy of $V_1$ and $B_2$ bounds a copy of $V_2$). 

Now $B_{\Sigma}(P)=P_1\sqcup P_2$ is (up to isotopy) as follows (see right of Figure \ref{fig:bingdouble_sliceball}).

\begin{itemize}
    \item $(S^4\times[-1,0], S^4\times[-1,0]\cap P_1\cup P_2)$ is the mirror of $(S^4\times[0,1], S^4\times[0,1]\cap P_1\cup P_2)$.
    \item $P_1\cap (S^4\times [0,1/3))=B_1\times[0,1/3)$,
    \item $P_1\cap (S^4\times\{1/3\})=(P_1\cup H_1\cup H_2)\times\{1/3\}$,
    \item $P_1\cap(S^4\times (1/3,1))=B_1'\times(1/3,1)$,
    \item $P_1\cap(S^4\times\{1\}))=V_1\times\{1\}$,
    \item $P_2\cap (S^4\times[0,1))=B_2\times[0,1)$,
    \item $P_2\cap (S^4\times\{1\})=V_2\times\{1\}$.
\end{itemize}
We smooth $P_1\sqcup P_2$ and obtain the desired restricted height function.
\end{proof}

\begin{figure}
\labellist
\pinlabel{$h$} at 2.5 -2
\pinlabel{$0$} at -2 10
\pinlabel{$\frac{1}{3}$} at -2 35
\pinlabel{$\frac{2}{3}$} at -2 60
\pinlabel{$1$} at -2 85
\pinlabel{\huge{$\cong$}} at 47 60
\pinlabel{$h$} at 115 -2
\pinlabel{$0$} at 110 10
\pinlabel{$\frac{1}{3}$} at 110 35
\pinlabel{$\frac{2}{3}$} at 110 60
\pinlabel{$1$} at 110 85
\pinlabel{\huge{$\cong$}} at 155.5 60
\pinlabel{$K$} at 16 19
\pinlabel{$D_1$} at 25 11
\pinlabel{$D_2$} at 45 11
\pinlabel{$V$} at 55 95
\pinlabel{$T$} at 55 70
\pinlabel{\textcolor{blue}{$\Sigma$}} at 87 90
\pinlabel{\textcolor{blue}{$C$}} at 38 -1
\pinlabel{\textcolor{gray}{$H$}} at 9 46
\pinlabel{$B_1$} at 124 11
\pinlabel{$B_1'$} at 122 71
\pinlabel{$B_2$} at 154 11
\pinlabel{$V_2$} at 202 85
\pinlabel{$B(T)$} at 199 70
\pinlabel{$V_1$} at 158 85
\pinlabel{\textcolor{gray}{$H_1,H_2$}} at 123 47.5
\endlabellist
    \includegraphics[width=135mm]{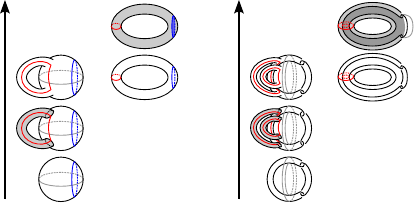}
    \vspace{.1in}

    \caption{The construction of Proposition~\ref{prop:bingmorse}. Left: the portion of the simple 3-knot $P$ contained in $h^{-1}[0,1]$. Right: the portion of the Bing double $B(P)=P_1\sqcup P_2$ contained in $h^{-1}[0,1]$.}\label{fig:bingdouble_sliceball}
\end{figure}
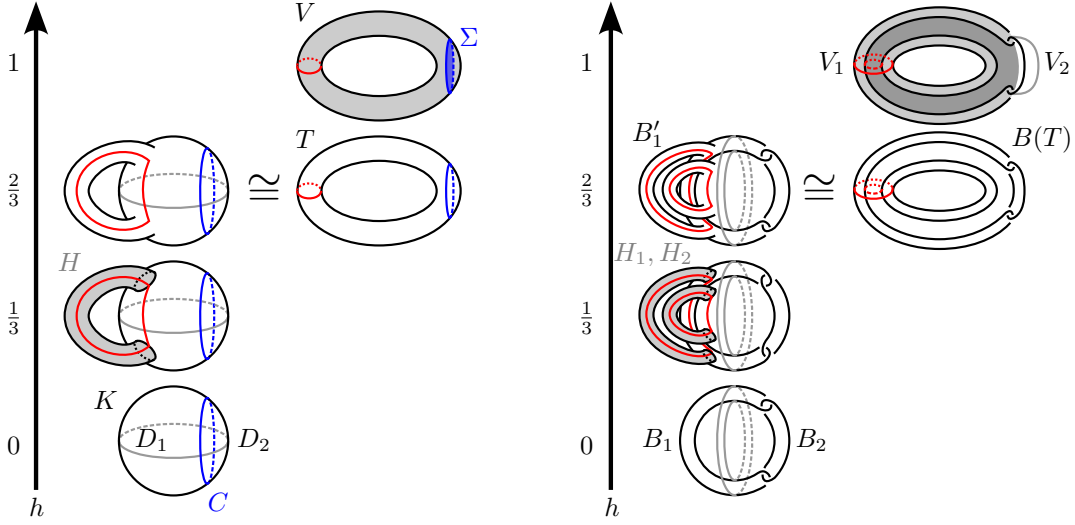

\begin{remark}\label{rem:satoh} By Satoh \cite{satoh}, Proposition~\ref{prop:bingmorse} applies to any twist-spun knot of a classical knot of unknotting number one. In particular, we may take $K$ to be the 2-twist spun trefoil.\end{remark}

\begin{remark}\label{rem:22}
    In Proposition~\ref{prop:bingmorse}, if we take the 2-sphere $\Sigma$ to separate the two index-1 critical points, then we can arrange for $h|_{P_i}$ to have one local minimum, two index-1 critical points, two index-2 critical points, and one local maximum for each $i=1,2$. 
\end{remark}

\section{Constructing splitting spheres}\label{sec:construct}

In this section we show how to construct distinct splitting spheres in the setting of Theorem~\ref{thm:main}. We first present a lemma that is analogous to \cite[Lemma 4.1]{handlebodypaper}.

\begin{lemma}\label{lem:bdryparallel}
    Let $H_1\sqcup H_2$ be a disjoint union of two 3-dimensional handlebodies properly embedded in $B^5$. Suppose that with respect to the radial height function $h:B^5\to\mathbb{R}$, both $h|_{H_1}$ and $h|_{H_2}$ have exactly one local minimum and no index-2 or index-3 critical points. Suppose also that $\partial H_1\sqcup \partial H_2$ is an unlink in $\partial B^5=S^4$. Then $H_1\sqcup H_2$ is smoothly boundary parallel.
\end{lemma}

\begin{proof}
Parameterize $h:B^5\to\mathbb{R}$ so that $h(\partial B^5)=1$ and $h$ maps the center of $B^5$ to $0$. Isotope $H_1\sqcup H_2$ so that both index-0 critical points of $h|_{(H_1\sqcup H_2)}$ lie in $h^{-1}(1/4)$, and all index-1 critical points lie in disjoint heights $t_1<\ldots< t_{m+n}$ in $h^{-1}(2/3,3/4)$.

Now $H_1\sqcup H_2$ intersects the 4-sphere $h^{-1}(1/2)$ in a 2-component unlink of 2-spheres $S_1\sqcup S_2$. As the value of $h$ increases, the surfaces $S_1\sqcup S_2$ are successively surgered along 3-dimensional 1-handles. Fix a basepoint of $B^5$ in $h^{-1}(1/2)$, so $\pi_1(h^{-1}(1/2)\setminus(S_1\sqcup S_2))=\langle a,b\rangle$, where $a$ is a meridian of $S_1$ and $b$ is a meridian of $S_2$.

Consider the effect of surgering $S_1$ along a 3-dimensional 1-handle at height $t_1$ whose core arc traces out an arc representing a word $w$ in the letters $a,b$. (Such an arc is completed to a loop by choosing a fixed whisker from $S_1$ to the basepoint; $w$ is thus defined only up to conjugation and multiplication by $a$ at either end.) We obtain the relation $waw^{-1}=a$ in the presentation of $\pi_1(h^{-1}(t_1+\epsilon)\setminus(H_1\sqcup H_2))$. Similarly if the handle instead has ends on $S_2$, then we obtain a relation of the form $wbw^{-1}=b$.

Let $w_1,\ldots, w_m$ be words represented by the cores of the 3-dimensional 1-handles along which $S_1$ is surgered as $h$ increases. Let $W_1,\ldots, W_n$ be words reprsented by the cores of the 3-dimensional 1-handles along which $S_2$ is surgered. Then we obtain a presentation of $\pi_1(h^{-1}(1)\setminus(H_1\sqcup H_2))=\pi_1(S^4\setminus\partial(H_1\sqcup H_2))$ as \[\langle a,b\mid w_1aw_1^{-1}a^{-1}=\cdots=w_maw_ma^{-1}=W_1bW_1^{-1}b^{-1}=\cdots=W_nbW_n^{-1}b^{-1}=1.\rangle.\]

On the other hand, by assumption $\pi_1(S^4\setminus\partial(H_1\sqcup H_2))$ is free on two generators. We conclude that $w_i,W_j=1$ for all $i,j$. Since the cores of the 3-dimensional 1-handles attached to $S_1,S_2$ are 1-dimensional and each $h^{-1}(t)$ is 4-dimensional, up to smooth isotopy preserving $h$ we can trivialize the 3-dimensional 1-handles.  Thus, $H_1\sqcup H_2$ is smoothly boundary-parallel.
\end{proof}

We are now ready to construct splitting spheres. We restate the main theorem again here.

\begin{mainthm}
    Let $L=L_1\sqcup L_2$ be a smooth 2-component unlink in $S^4$ with $g(L_1)\ge 4$. There exist smooth splitting spheres $S_1,S_2$ for $L$ so that $S_1,S_2$ are not topologically isotopic in $S^4\setminus L$. 
\end{mainthm}

We include a schematic of the proof of Theorem~\ref{thm:main} in the  case $g(L_1)=4,g(L_2)=0$ in Figure \ref{fig:schematic_proof}.

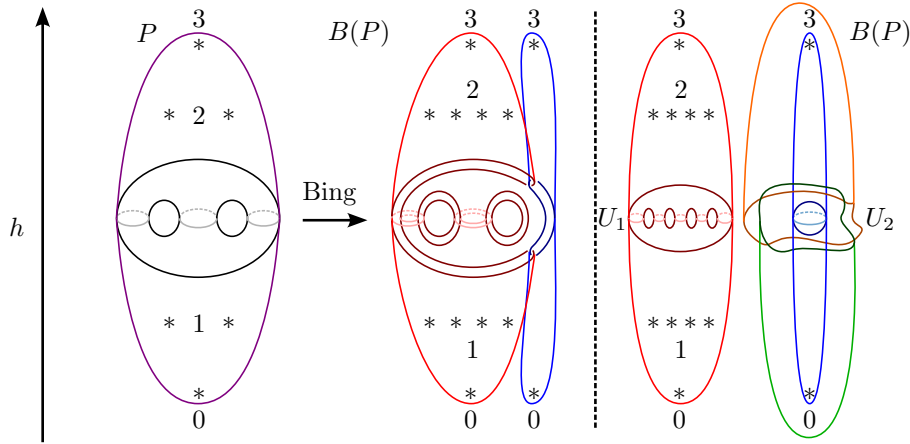
\begin{figure}
\labellist
\pinlabel{$P$} at 30 110
\pinlabel{$B(P)$} at 87 110
\pinlabel{$B(P)$} at 227 110
\pinlabel{$U_1$} at 155 60
\pinlabel{$U_2$} at 227 60
\pinlabel{Bing} at 79 67
\pinlabel{$h$} at -5 58

\pinlabel{$\ast$} at 44 13
\pinlabel{$\ast$} at 36 32
\pinlabel{$\ast$} at 52 32
\pinlabel{$\ast$} at 44 107
\pinlabel{$\ast$} at 52 88
\pinlabel{$\ast$} at 36 88
\pinlabel{$0$} at 44 6
\pinlabel{$1$} at 44 32
\pinlabel{$3$} at 44 114
\pinlabel{$2$} at 44 88

\pinlabel{$\ast$} at 134 13
\pinlabel{$0$} at 134 6
\pinlabel{$\ast$} at 134 107
\pinlabel{$3$} at 134 114

\pinlabel{$\ast$} at 117 13
\pinlabel{$0$} at 117 6
\pinlabel{$\ast$} at 117 107
\pinlabel{$3$} at 117 114
\pinlabel{$\ast$} at 106 32
\pinlabel{$\ast$} at 113 32
\pinlabel{$\ast$} at 120 32
\pinlabel{$\ast$} at 127 32
\pinlabel{$1$} at 117.5 25
\pinlabel{$\ast$} at 106 88
\pinlabel{$\ast$} at 113 88
\pinlabel{$\ast$} at 120 88
\pinlabel{$\ast$} at 127 88
\pinlabel{$2$} at 117.5 95

\pinlabel{$\ast$} at 173 13
\pinlabel{$0$} at 173 6
\pinlabel{$\ast$} at 173 107
\pinlabel{$3$} at 173 114
\pinlabel{$\ast$} at 208 13
\pinlabel{$0$} at 208 6
\pinlabel{$\ast$} at 208 107
\pinlabel{$3$} at 208 114
\pinlabel{$\ast$} at 166 32
\pinlabel{$\ast$} at 171 32
\pinlabel{$\ast$} at 176 32
\pinlabel{$\ast$} at 181 32
\pinlabel{$1$} at 173.5 25
\pinlabel{$\ast$} at 166 88
\pinlabel{$\ast$} at 171 88
\pinlabel{$\ast$} at 176 88
\pinlabel{$\ast$} at 181 88
\pinlabel{$2$} at 173.5 95
\endlabellist
    \includegraphics[width=110mm]{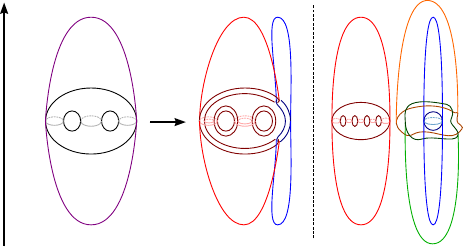}
    \caption{The proof of Theorem~\ref{thm:main}. We first construct a nontrivial, simple 3-knot $P$ and its Bing double $B(P)$ with the property that $h|_{B(P)}$ is as in Proposition~\ref{prop:bingmorse} (here we indicate a critical point with $\ast$ and write the index). Then by Lemma~\ref{lem:bdryparallel}, the top and bottom halves (that is, on either side of $S^4\times 0$) of $B(P)$ are smoothly boundary-parallel, and hence splittable by properly embedded smooth 4-balls in $h^{-1}(-\infty,0]$ and $h^{-1}[0,\infty)$. These 4-balls intersect $h^{-1}(0)$ in smooth splitting spheres for an unlink $L$ of a genus-0 and a genus-4 surface. These splitting spheres cannot be topologically isotopic in $S^4\setminus L$, or else the link $B(P)$ would be split, contradicting Lemma~\ref{lemma:bingnotsplit}.}\label{fig:schematic_proof}
\end{figure}
\begin{proof}
Let $U_1\sqcup U_2$ be the 2-component unlink in $S^4$ with $U_1$ a genus-4 surface and $U_2$ a 2-sphere. Let $K$ be the $2$-twist spun trefoil.  By Proposition~\ref{prop:bingmorse} and Remark \ref{rem:satoh}, there is a standard height function $h:S^5\to\mathbb{R}$ and a smooth simple 3-knot $P\subset S^5$ so that $P$ intersects $h^{-1}(0)\cong S^4$ in a copy of $K$ and $B(P)=P_1\sqcup P_2$ intersects $h^{-1}(0)$ in $B(K)=K_1\sqcup K_2$ (with $P_i\cap S^4=K_i$). Moreover, the height function $h$ on $S^5$ restricts to $P$ in a Morse function with one local minimum, two index-1 points, two index-2 points, and one local maximum; and restricts to $B(P)$ in a Morse function with one local minimum and local maximum on each $P_1,P_2$ in addition to four index-1 and four index-2 critical points on $P_1$. By Stoltzfus \cite{stoltzfus}, $P$ is a nontrivial 3-knot.

Smoothly isotope $P$ and $B(P)$ to reorder the critical points of $h|_P$ and $h|_{B(P)}$ so that index-0 and index-1 points are in $h^{-1}(-\infty,0)$ and index-2 and index-3 points are in $h^{-1}(0,\infty).$

Now $P$ intersects $h^{-1}(0)$ in a trivial genus-2 surface $\Sigma$ and $B(P)$ intersects $h^{-1}(0)$ in a Bing double $B_\gamma(\Sigma)$ , where $\gamma$ is a curve bounding a disk in $\Sigma$.  By Proposition~\ref{prop:bingunknotisunlink}, $B_\gamma(\Sigma)$ is smoothly isotopic to the unlink $U_1\sqcup U_2$. Thus, $B(P)$ intersects $h^{-1}(0)$ in a copy of $U_1\sqcup U_2$, with $P_i\cap h^{-1}(0)=U_i$.

Consider the 2-component 3-manifold $V_1=B(P)\cap h^{-1}(\infty,0]$ properly embedded in $h^{-1}(-\infty,0]\cong B^5$. By Lemma~\ref{lem:bdryparallel}, $V_1$ is boundary-parallel. Let $S_1$ be a splitting sphere for $U_1\sqcup U_2$ in $S^4=h^{-1}(0)$ so that $V_1$ can be pushed into $S^4=h^{-1}(0)$ in the complement of $S_1$. Then there exists a trivial 4-ball $B_1$ properly embedded in $h^{-1}(\infty,0]$ with boundary $S_1$ that is disjoint from $V_1$. Similarly, apply Lemma~\ref{lem:bdryparallel} to the 2-component 3-manifold $V_2=B(P)\cap h^{-1}[0,\infty)$ properly embedded in $h^{-1}[0,\infty)\cong B^5$ to obtain another splitting sphere $S_2$ for $U_1\sqcup U_2$ in $h^{-1}(0)$ that bounds a 4-ball $B_2$ in $h^{-1}[0,\infty)$ in the complement of $V_2$.

Now suppose that there is a topological isotopy of $S^4\setminus(U_1\sqcup U_2)$ that takes $S_1$ to $S_2$. Then we can extend this isotopy over $h^{-1}(-\infty,0]\setminus P$. This isotopy takes the ball $B_1$ to some other 4-ball $B_1'$ whose boundary is $S_2$. Then $B_1'\cup B_2$ forms a splitting sphere for $B(P)$. But since $P$ is nontrivial, by Lemma~\ref{lemma:bingnotsplit} the Bing double $B(P)$ is not split and hence we obtain a contradiction. We thus conclude that the splitting spheres $S_1,S_2$ for $U_1\sqcup U_2$ are not topologically isotopic in $S^4\setminus (U_1\sqcup U_2)$, completing the proof in the case that $g(L_1)=4, g(L_2)=0$ 
(here setting $L_i=U_i$).

Now consider perturbing either component $P_1,P_2$ of $B(P)$ to introduce a new index-1 critical point of $h|_{P_i}$ in $h^{-1}(-\infty,0]$ and a new index-2 critical point of $h|_{P_i}$ in $h^{-1}[0,\infty)$. Now $S_1,S_2$ are splitting spheres for an unlink obtained by adding a tube to either component of $U_1\sqcup U_2$, and again cannot be topologically isotopic in the unlink complement or else $B(P)$ would be split. By repeating this argument, we obtain non-isotopic splitting spheres in the complement of $L=L_1\sqcup L_2$.
\end{proof}

\begin{remark}\label{rem:proof22}
    We can adapt the proof of Theorem~\ref{thm:main} to the case when $(g(L_1),g(L_2))$ lies in the set $\{(2,2),(2,3),(3,3)\}$ using the observation of Remark \ref{rem:22}: the Morse function on $B(P)$ as constructed in Proposition~\ref{prop:bingmorse} can be taken to have two index-1 critical points on each component, so that the unlink cross-section of $B(P)$ is a split union of two genus-2 surfaces, rather than a genus-4 surface and a 2-sphere.
\end{remark}

\begin{corollary}
Let $K_1\sqcup K_2$ be a split, doubly slice link of 2-spheres in $S^4$. Let $\Sigma_1$ be a genus-4 surface obtained by trivially stabilizing $K_1$ four times. Then the link $\Sigma_1\sqcup K_2$ has two smooth splitting spheres that are not topologically isotopic in the complement of $\Sigma_1\sqcup K_2$.
\end{corollary}

\begin{proof}
We repeat the proof of Theorem~\ref{thm:main} in the case $g(L_1)=4,g(L_2)=0$. Let $Q_1,Q_2$ be small unknotted 3-knots meeting $S^4$ in $K_1, K_2$, respectively. Then we can connect-sum on $Q_1,Q_2$ to $P_1,P_2$ respectively and continue the argument of Theorem~\ref{thm:main} to obtain two non-isotopic splitting spheres for $\Sigma_1\sqcup K_2$.
\end{proof}

As a consequence of the proof of Theorem~\ref{thm:main}, we are now ready to prove Theorem~\ref{thm:binghandlebody}.

\begin{binghandlebody}
Let $m,n$ be non-negative integers with $m\ge 4$. There exist two smooth, properly embedded links $H_1\sqcup H_2,\thicktilde{H}_1\sqcup \thicktilde{H}_2$ of handlebodies in $B^5$ with $g(H_1)=g(\thicktilde{H_1})=m$, $g(H_2)=g(\thicktilde{H}_2)=n$ so that the following are true.
    \begin{itemize}
        \item The 2-component links $H_1\sqcup H_2$ and $\thicktilde{H}_1\sqcup \thicktilde{H}_2$ are each smoothly boundary-parallel in $B^5$.
        \item The boundaries agree componentwise, i.e., $\partial H_1=\partial \thicktilde{H}_1, \partial H_2=\partial\thicktilde{H}_2$.
        \item For each $i=1,2$, the handlebodies $H_i,\thicktilde{H}_i$ are smoothly isotopic rel.\ boundary. 
         \item The 2-component links $H_1\sqcup H_2$ and $\thicktilde{H}_1\sqcup \thicktilde{H}_2$ are not topologically isotopic rel.\ boundary.
    \end{itemize}
\end{binghandlebody}
\begin{proof}
  As in the proof of Theorem~\ref{thm:main} (and indeed, the proof of the main theorem of \cite{handlebodypaper}), we obtain a smooth, nontrivial 3-knot $P$ so that the standard height function $h:S^5\to\mathbb{R}$ when restricted to $B(P)=P_1\sqcup P_2$ has one local minimum on each component, four index-1 critical points on $P_1$ four index-2 critical points on $P_1$, one local maximum on each component and no other critical points.

We obtain a new 3-link as follows (see Figure \ref{fig:schematic_proof_handlebody} for a schematic). Smoothly translate $P_2$ while preserving $h|_{P_2}$, until the resulting 3-knot $P_2'$ is split from $P_1$. Let $\Sigma_1,\Sigma_2,\Sigma_2'$ denote the intersections of $P_1,P_2,P'_2$ (respectively) with $h^{-1}(0)$. Then $\Sigma_2,\Sigma'_2$ are each 2-spheres that bound balls in the complement of $\Sigma_1$, and hence there is a smooth isotopy of $S^4$ fixing $\Sigma_1$ pointwise and taking $\Sigma'_2$ to $\Sigma_2$. Smoothly extend this isotopy over $S^5$, and let $\thicktilde{P}_1,\thicktilde{P}_2$ be the images of $P_1,P_2'$ of this isotopy. Then $\thicktilde{P}_1\sqcup \thicktilde{P}_2$ is a smooth 3-link intersecting $h^{-1}(0)$ with $\thicktilde{P}_1\cap h^{-1}(0)=\Sigma_1$, $\thicktilde{P}_2\cap h^{-1}(0)=\Sigma_2$. 
    
Let $W^T, W^B$ denote the 5-balls $h^{-1}[0,\infty)$ and $h^{-1}(-\infty,0]$, respectively. For $i=1,2$ and $*\in\{T,B\}$, let $H^*_i:=P_i\cap W^*$ and $\thicktilde{H}^*_i=\thicktilde{P}_i\cap W^*$. By Lemma~\ref{lem:bdryparallel}, the links $H^*_1\sqcup H^*_2$ and $\thicktilde{H}^*_1\sqcup \thicktilde{H}^*_2$ are each boundary-parallel in $W^*$. By construction, the genus-4 handlebodies $H_1^*,\thicktilde{H}_1^*$ are smoothly isotopic rel.\ boundary. The 3-balls $H_2^*,\thicktilde{H}_2^*$ are boundary-parallel with the same boundary and hence are isotopic rel.\ boundary in $W^*$ by \cite{hartman}. However, since $P_1\sqcup P_2$ is not split and $\thicktilde{P}_1\sqcup \thicktilde{P}_2$ is split, we conclude that either the pair $(H_1^T\sqcup H_2^T, \thicktilde{H}_1^T\sqcup \thicktilde{H}_2^T)$ are not topologically isotopic rel.\ boundary in $W^T$ or the pair $(H_1^B\sqcup H_2^B, \thicktilde{H}_1^B\sqcup \thicktilde{H}_2^B)$ are not topologically isotopic rel.\ boundary in $W^B$.

Perturbing $B(P)$ to add additional $(1,2)$-pairs to $h|_{B(P)}$ allows us to increase the genus of the resulting pair of non-isotopic handlebody links, as before. 
\end{proof}

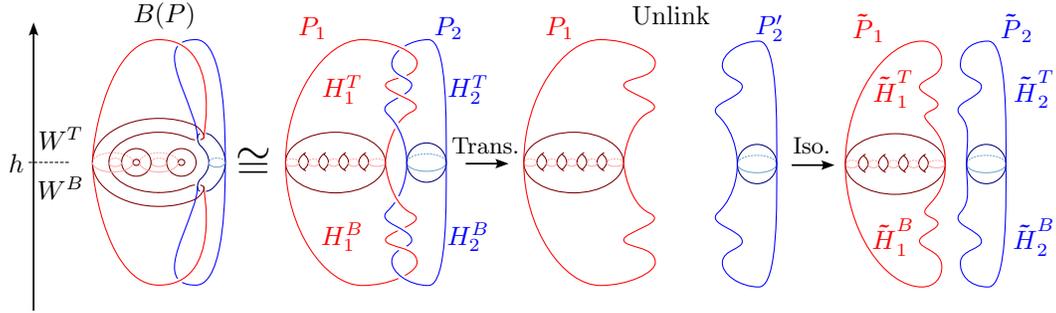
\begin{figure}
\labellist
\pinlabel{$h$} at -5 60
\pinlabel{$W^T$} at 13 70
\pinlabel{$W^B$} at 13 50

\pinlabel{\huge{$\cong$}} at 91 60
\pinlabel{\small{Trans.}} at 185 67
\pinlabel{\small{Iso.}} at 317 67

\pinlabel{$B(P)$} at 55 120

\pinlabel{Unlink} at 260 120

\pinlabel{\textcolor{red}{$P_1$}} at 115 115
\pinlabel{\textcolor{blue}{$P_2$}} at 170 115

\pinlabel{\textcolor{red}{$H_1^T$}} at 127 90
\pinlabel{\textcolor{blue}{$H_2^T$}} at 178 90
\pinlabel{\textcolor{red}{$H_1^B$}} at 127 30
\pinlabel{\textcolor{blue}{$H_2^B$}} at 178 30

\pinlabel{\textcolor{red}{$P_1$}} at 215 115
\pinlabel{\textcolor{blue}{$P_2'$}} at 300 115

\pinlabel{\textcolor{red}{$\thicktilde{P}_1$}} at 340 115
\pinlabel{\textcolor{blue}{$\thicktilde{P}_2$}} at 400 115

\pinlabel{\textcolor{red}{$\thicktilde{H}_1^T$}} at 350 90
\pinlabel{\textcolor{blue}{$\thicktilde{H}_2^T$}} at 407 90
\pinlabel{\textcolor{red}{$\thicktilde{H}_1^B$}} at 350 30
\pinlabel{\textcolor{blue}{$\thicktilde{H}_2^B$}} at 407 30
\endlabellist
\vspace{.1in}
\includegraphics[width=130mm]{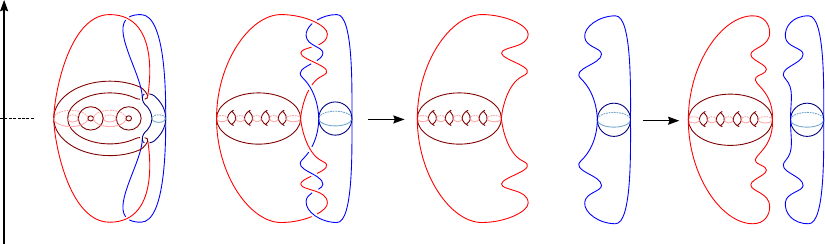}
\vspace{.1in}
    \caption{A schematic of the construction of Theorem~\ref{thm:binghandlebody}. First and second: the Bing double $B(P)=P_1\sqcup P_2$ of a nontrivial 3-knot $P$. We include the second frame to emphasize that $B(P)$ intersects $h^{-1}(0)$ in an unlink. Third: we translate $P_2$ (while preserving $h|_{P_2}$) to obtain a 3-knot $P_2'$ split from $P_1$. Fourth: we smoothly isotope $h^{-1}(0)$ (while keeping the intersection of $P_1$ with $h^{-1}(0)$ fixed) and extend over all of $S^5$ so as to isotope $P_1,P_2'$ to 3-knots $\thicktilde{P}_1,\thicktilde{P}_2$ whose intersections with $h^{-1}(0)$ agree with those of $P_1,P_2$. For $i\in\{1,2\}$ and $*\in\{B,T\}$, the handlebodies $H_i^*$ and $\thicktilde{H}_i^*$ are smoothly isotopic rel.\ boundary in the 5-ball $W^*$.}\label{fig:schematic_proof_handlebody}
\end{figure}

\begin{remark}
 One can remove the ambiguity from the final conclusion in the proof of Theorem~\ref{thm:binghandlebody} via repeated applications of \cite{hirose} (see discussion in \cite{handlebodypaper}), but this is tedious and not particularly enlightening in practice -- though \cite{hirose} is in principle constructive, in practice it is very difficult to implement the construction.
\end{remark}

\bibliographystyle{alpha}
\bibliography{biblio}

\begin{thebibliography}{HKM23}

\bibitem[Alt78]{althoen}
Steven~C. Althoen.
\newblock A van {K}ampen theorem for {$\pi \sb{2}$}.
\newblock {\em J. Pure Appl. Algebra}, 10(3):257--269, 1977/78.

\bibitem[BG]{budneygabai}
Ryan {Budney} and David {Gabai}.
\newblock {Knotted 3-balls in $S^4$}.
\newblock ArXiv:1912.09029 [math.GT], Apr. 2021.

\bibitem[Har]{hartman}
Daniel Hartman.
\newblock {Unknotting 3-Balls in the 5–Ball}.
\newblock ArXiv:2206.11243 [math.GT], June 2022.

\bibitem[Hir02]{hirose}
Susumu Hirose.
\newblock On diffeomorphisms over surfaces trivially embedded in the 4-sphere.
\newblock {\em Algebr. Geom. Topol.}, 2:791--824, 2002.

\bibitem[HKM23]{handlebodypaper}
Mark Hughes, Seungwon Kim, and Maggie Miller.
\newblock Knotted handlebodies in the 4-sphere and 5-ball.
\newblock {\em {\emph{to appear in }}J. Eur. Math. Soc.}, 2023.

\bibitem[IKMT]{konno}
Nobuo Iida, Hokuto Konno, Anubhav Mukherjee, and Masaki Taniguchi.
\newblock {Diffeomorphisms of 4-manifolds with boundary and exotic embeddings}.
\newblock ArXiv:2203.14878 [math.GT], Mar. 2022.

\bibitem[KS17]{kreck}
Matthias Kreck and Yang Su.
\newblock On 5-manifolds with free fundamental group and simple boundary links
  in {$S^5$}.
\newblock {\em Geom. Topol.}, 21(5):2989--3008, 2017.

\bibitem[Lau73]{laudenbach}
F.~Laudenbach.
\newblock Sur les {$2$}-sph\`eres d'une vari\'{e}t\'{e} de dimension {$3$}.
\newblock {\em Ann. of Math. (2)}, 97:57--81, 1973.

\bibitem[LP72]{laudenbachpoenaru}
Fran\c{c}ois Laudenbach and Valentin Po\'{e}naru.
\newblock A note on {$4$}-dimensional handlebodies.
\newblock {\em Bull. Soc. Math. France}, 100:337--344, 1972.

\bibitem[Qui82]{quinn}
Frank Quinn.
\newblock Ends of maps. {III}. {D}imensions {$4$} and {$5$}.
\newblock {\em J. Differential Geometry}, 17(3):503--521, 1982.

\bibitem[Sat04]{satoh}
Shin Satoh.
\newblock A note on unknotting numbers of twist-spun knots.
\newblock {\em Kobe J. Math.}, 21(1-2):71--82, 2004.

\bibitem[Sha68]{shaneson}
Julius~L. Shaneson.
\newblock Embeddings with codimension two of spheres in spheres and
  {$H$}-cobordisms of {$S\sp{1}\times S\sp{3}$}.
\newblock {\em Bull. Amer. Math. Soc.}, 74:972--974, 1968.

\bibitem[Sta63]{stallings}
John Stallings.
\newblock On topologically unknotted spheres.
\newblock {\em Ann. of Math. (2)}, 77:490--503, 1963.

\bibitem[Sto79]{stoltzfus}
Neal~W. Stoltzfus.
\newblock Isometries of inner product spaces and their geometric applications.
\newblock In {\em Geometric topology ({P}roc. {G}eorgia {T}opology {C}onf.,
  {A}thens, {G}a., 1977)}, pages 527--541. Academic Press, New York-London,
  1979.

\end{thebibliography}
\end{document}